\documentclass{article}

\usepackage{amsmath,amsbsy,amssymb,latexsym,amsthm}

\usepackage{graphicx}
\usepackage{psfrag}

\usepackage{cite}

\newtheorem{Th}{Theorem}[section]
\newtheorem{Cor}[Th]{Corollary}

\newtheorem{Lemma}[Th]{Lemma}

\theoremstyle{definition}
\newtheorem{Example}[Th]{Example}
\newtheorem{df}[Th]{Definition}

\theoremstyle{remark}
\newtheorem{Remark}[Th]{Remark}

\newcommand{\R}{\mathbb{R}}
\newcommand{\Z}{\mathbb{Z}}

\newcommand{\T}{\mathbb{T}}

\begin{document}

\title{Backward variational approach on time scales
with an action depending on the free endpoints\thanks{Submitted 17-Oct-2010; 
revised 18-Dec-2010; accepted 4-Jan-2011; 
for publication in \emph{Zeitschrift f\"{u}r Naturforschung A}}}

\author{Agnieszka B. Malinowska${}^{a}$\\
        \texttt{abmalinowska@ua.pt}
        \and
        Delfim F. M. Torres${}^{b, }$\thanks{Corresponding author.
        Email: delfim@ua.pt; Tel: +351 234370668; Fax: +351 234370066}\\
        \texttt{delfim@ua.pt}}

\date{${}^a$Faculty of Computer Science\\
Bia{\l}ystok University of Technology\\
15-351 Bia{\l}ystok, Poland\\[0.3cm]
${}^b$Department of Mathematics\\
University of Aveiro\\
3810-193 Aveiro, Portugal}

\maketitle


\begin{abstract}
We establish necessary optimality conditions for variational problems
with an action depending on the free endpoints.
New transversality conditions are also obtained.
The results are formulated and proved
using the recent and general theory
of time scales via the backward nabla differential operator.
\end{abstract}

\smallskip

\noindent \textbf{PACS 2010:}
02.30.Xx, 02.30.Yy.

\smallskip


\smallskip


\smallskip



\section{Introduction}

Physics and Control on an arbitrary time scale is an area of strong
current research that unifies discrete, continuous, and quantum results and
generalize the theory to more complex domains \cite{MyID:162,BarPaw,withEwaP:avoidance}.
The new calculus on time scales has been applied, among others, in physics and control
of population, quantum calculus, economics, communication networks,
and robotic control (see \cite{SSW} and references therein). The
variational approach on time scales is  a fertile area
under strong current research
\cite{comNatyZbig,des:ts:cv,Tangier07,rui:ho,china-Xuzhou,comNatyBasia:infHorizon,MyID:183,comNataliaPoland07}.
In this paper we study problems in Lagrange form
with an action functional and a velocity vector
without boundary conditions $x(a)$ and $x(b)$.
The considered problems are more general
because of the dependence of the Hamiltonian on
$x(a)$ and $x(b)$. Such possibility is not covered
by the literature. Our study is
done using the nabla approach to time scales, which seems promising
with respect to applications (see, \textrm{e.g.},
\cite{iso,Atici:2006,Atici:2008}). This work is motivated by
the recent advancements obtained in \cite{withAlanZinober,Basia:post_doc_Aveiro:2}
about necessary optimality conditions for the
problem of the calculus of variations with a free endpoint $x(T)$
but whose Lagrangian depends explicitly on $x(T)$. Such problems
seem to have important implications in physical applications \cite{withAlanZinober}.
In contrast with \cite{withAlanZinober,Basia:post_doc_Aveiro:2},
we adopt here a backward perspective, which has proved useful,
and sometimes more natural and preferable, with respect to applications
\cite{iso,Atici:2006,Atici:2008,MT,MyID:177}.
The advantage of the backward approach here promoted becomes
evident when one considers that the
time scales analysis can also have important
implications for numerical analysts, who often prefer
backward differences rather than forward differences to
handle their computations due to practical implementation reasons
and also for better stability properties of implicit discretizations
\cite{MR2436490,MyID:177}.

The paper is organized as follows. Section~\ref{sec:prel} presents
the necessary definitions and concepts of the calculus on time
scales; our results are formulated, proved, and illustrated through
examples in Section~\ref{sec:mr}. Both Lagrangian
(Section~\ref{sec:mr:lag}) and Hamiltonian
(Section~\ref{sec:mr:ham}) approaches are considered. Main results
of the paper include necessary optimality conditions with new
transversality conditions (Theorems~\ref{main} and
\ref{hamiltonian}) that become sufficient under appropriate
convexity assumptions (Theorem~\ref{sc}).


\section{Time Scales Calculus}
\label{sec:prel}

For a general introduction to the calculus on time scales
we refer the reader to the books \cite{BP1,BP2}.
Here we only give those notions and results needed in the sequel.
In particular we are interested in the
backward nabla differential approach to time scales \cite{MyID:177}.
As usual $\mathbb{R}$,
$\mathbb{Z}$, and $\mathbb{N}$ denote, respectively, the set of
real, integer, and natural numbers.

A \emph{time scale} $\mathbb{T}$ is an arbitrary nonempty closed
subset of $\mathbb{R}$. Thus, $\mathbb{R}$, $\mathbb{Z}$, and
$\mathbb{N}$, are trivial examples of times scales. Other examples
of times scales are: $[-1,4] \bigcup \mathbb{N}$, $h\mathbb{Z}:=\{h
z | z \in \mathbb{Z}\}$ for some $h>0$, $q^{\mathbb{N}_0}:=\{q^k | k
\in \mathbb{N}_0\}$ for some $q>1$, and the Cantor set. We assume
that a time scale $\mathbb{T}$ has the topology that it inherits
from the real numbers with the standard topology.

The \emph{forward jump operator}
$\sigma:\mathbb{T}\rightarrow\mathbb{T}$ is defined by
$\sigma(t)=\inf{\{s\in\mathbb{T}:s>t}\}$ if $t\neq \sup \mathbb{T}$,
and $\sigma(\sup \mathbb{T})=\sup \mathbb{T}$. The \emph{backward
jump operator} $\rho:\mathbb{T}\rightarrow\mathbb{T}$ is defined by
$\rho(t)=\sup{\{s\in\mathbb{T}:s<t}\}$ if $t\neq \inf \mathbb{T}$,
and $\rho(\inf \mathbb{T})=\inf \mathbb{T}.$

A point $t\in\mathbb{T}$ is called \emph{right-dense},
\emph{right-scattered}, \emph{left-dense} and \emph{left-scattered}
if $\sigma(t)=t$, $\sigma(t)>t$, $\rho(t)=t$, and $\rho(t)<t$,
respectively. We say that $t$ is \emph{isolated} if
$\rho(t)<t<\sigma(t)$, that $t$ is \emph{dense} if
$\rho(t)=t=\sigma(t)$. The \emph{(backward) graininess function}
$\nu:\mathbb{T}\rightarrow[0,\infty)$ is defined by $\nu(t)=t -
\rho(t)$, for all $t\in\mathbb{T}$. Hence, for a given $t$, $\nu(t)$
measures the distance of $t$ to its left neighbor. It is clear that
when $\mathbb{T}=\mathbb{R}$ one has $\sigma(t)=t=\rho(t)$, and
$\nu(t)=0$ for any $t$. When $\mathbb{T}=\mathbb{Z}$,
$\sigma(t)=t+1$, $\rho(t)=t-1$, and $\nu(t)=1$ for any $t$.

In order to introduce the definition of nabla derivative, we define
a new set $\mathbb{T}_\kappa$ which is derived from $\mathbb{T}$ as
follows: if  $\mathbb{T}$ has a right-scattered minimum $m$, then
$\mathbb{T}_\kappa=\mathbb{T}\setminus\{m\}$; otherwise,
$\mathbb{T}_\kappa= \mathbb{T}.$

\begin{df}
We say that a function $f:\mathbb{T}\rightarrow\mathbb{R}$ is
\emph{nabla differentiable} at $t\in\mathbb{T}_\kappa$ if there is a
number $f^{\nabla}(t)$ such that for all $\varepsilon>0$ there
exists a neighborhood $U$ of $t$ (\textrm{i.e.},
$U=]t-\delta,t+\delta[\cap\mathbb{T}$ for some $\delta>0$) such that
$$|f(\rho(t))-f(s)-f^{\nabla}(t)(\rho(t)-s)|
\leq\varepsilon|\rho(t)-s|,\mbox{ for all $s\in U$}.$$ We call
$f^{\nabla}(t)$ the \emph{nabla derivative} of $f$ at $t$. Moreover,
we say that $f$ is \emph{nabla differentiable} on $\mathbb{T}$
provided $f^{\nabla}(t)$ exists for all $t \in \mathbb{T}_\kappa$.
\end{df}

\begin{Th}{\rm (Theorem~8.39 in \cite{BP1})}
\label{propriedades derivada} Let $\mathbb{T}$ be a time scale,
$f:\mathbb{T}\rightarrow\mathbb{R}$, and $t\in\mathbb{T}_\kappa$.
If $f$ is nabla differentiable at $t$, then $f$ is
continuous at $t$. If $f$ is continuous at $t$ and $t$ is left-scattered,
then $f$ is nabla differentiable at $t$ and
$f^{\nabla}(t)=\frac{f(t)-f(\rho(t))}{t-\rho(t)}$.
If $t$ is left-dense, then $f$ is nabla differentiable
at $t$ if and only if the limit
$\lim_{s\rightarrow t} \frac{f(t)-f(s)}{t-s}$
exists as a finite number. In this case,
$f^\nabla(t)=\lim_{s\rightarrow t} \frac{f(t)-f(s)}{t-s}$.
If $f$ is nabla differentiable at $t$, then
$f(\rho(t))=f(t)-\nu(t)f^\nabla(t)$.
\end{Th}

\begin{Remark}
When $\mathbb{T}=\mathbb{R}$, then $f:\mathbb{R} \rightarrow
\mathbb{R}$ is nabla differentiable  at $t \in \mathbb{R}$ if and
only if $f^{\nabla}(t)= \lim_{s\rightarrow
t}\frac{f(t)-f(s)}{t-s}$ exists, \textrm{i.e.}, if and only if $f$
is differentiable at $t$ in the ordinary sense. When
$\mathbb{T}=\mathbb{Z}$, then $f:\mathbb{Z} \rightarrow \mathbb{R}$
is always nabla differentiable at $t \in \mathbb{Z}$ and
$f^{\nabla}(t)=\frac{f(t)-f(\rho(t))}{t-\rho(t)}= f(t)-f(t-1)=:\nabla
f(t)$, \textrm{i.e.}, $\nabla$ is the usual backward difference
operator defined by the last equation above. For any time scale
$\mathbb{T}$, when $f$ is a constant, then $f^{\nabla}=0$; if
$f(t)=k t$ for some constant $k$, then $f^{\nabla}=k$.
\end{Remark}

In order to simplify expressions, we denote the composition $f\circ
\rho$ by $f^{\rho}$.

\begin{Th}{\rm (Theorem~8.41 in \cite{BP1})}
Suppose $f,g:\mathbb{T}\rightarrow\mathbb{R}$ are nabla
differentiable at $t\in\mathbb{T}_\kappa$. Then,
the sum $f+g:\mathbb{T}\rightarrow\mathbb{R}$ is nabla
differentiable at $t$ and $(f+g)^{\nabla}(t)=f^{\nabla}(t) +
g^{\nabla}(t)$; for any constant $\alpha$, $\alpha
f:\mathbb{T}\rightarrow\mathbb{R}$ is nabla differentiable at $t$
and $(\alpha f)^{\nabla} (t)=\alpha f^{\nabla}(t)$;
the product $fg:\mathbb{T}\rightarrow\mathbb{R}$ is
nabla differentiable at $t$ and
$(fg)^{\nabla}(t) = f^{\nabla}(t)g(t) + f^{\rho}(t)g^{\nabla}(t)
= f^{\nabla}(t)g^{\rho}(t)+ f(t)g^{\nabla}(t)$.
\end{Th}

\begin{df}
Let $\mathbb{T}$ be a time scale,
$f:\mathbb{T}\rightarrow\mathbb{R}$. We say that function $f$ is
\emph{$\nu$-regressive} if $1-\nu(t)f(t)\neq 0$ for all $\in
\mathbb{T}_\kappa$.
\end{df}

\begin{df}
A function $F:\mathbb{T}\rightarrow\mathbb{R}$ is called a
\emph{nabla antiderivative} of $f:\mathbb{T}\rightarrow\mathbb{R}$
provided $F^{\nabla}(t)=f(t)$ for all $t \in \mathbb{T}_\kappa$.
In this case we define the \emph{nabla integral} of $f$ from $a$ to
$b$ ($a,b \in \mathbb{T}$) by $\int_{a}^{b}f(t)\nabla t:=F(b)-F(a)$.
\end{df}

In order to present a class of functions that possess a nabla
antiderivative, the following definition is introduced.

\begin{df}
Let $\mathbb{T}$ be a time scale,
$f:\mathbb{T}\rightarrow\mathbb{R}$. We say that function $f$ is
\emph{ld-continuous} if it is continuous at left-dense points
and its right-sided limits exist (finite) at all right-dense points.
\end{df}

\begin{Th}{\rm (Theorem~8.45 in \cite{BP1})}
\label{antiderivada} Every ld-continuous function has a nabla
antiderivative. In particular, if $a \in \mathbb{T}$, then the
function $F$ defined by
$F(t)= \int_{a}^{t}f(\tau)\nabla\tau$, $t \in \mathbb{T}$,
is a nabla antiderivative of $f$.
\end{Th}

The set of all ld-continuous functions
$f:\mathbb{T}\rightarrow\mathbb{R}$ is denoted by
$C_{\textrm{ld}}(\mathbb{T}, \mathbb{R})$, and the set of all nabla
differentiable functions with ld-continuous derivative by
$C_{\textrm{ld}}^1(\mathbb{T}, \mathbb{R})$.

\begin{Th}{\rm (Theorem~8.46 in \cite{BP1})}
\label{propriedade integral} If $f \in C_{\textrm{ld}}(\mathbb{T},
\mathbb{R})$ and $t \in \mathbb{T}_\kappa$, then
$\int_{\rho(t)}^{t}f(\tau)\nabla\tau=\nu(t)f(t)$.
\end{Th}

\begin{Th}{\rm (Theorem~8.47 in \cite{BP1})}
\label{propriedades nabla integral} If $a$, $b$,
$c \in \mathbb{T}$, $a
\le c \le b$, $\alpha \in \mathbb{R}$, and $f,g \in
C_{\textrm{ld}}(\mathbb{T}, \mathbb{R})$, then
$\int_{a}^{b}\left(f(t) + g(t) \right)
    \nabla t= \int_{a}^{b}f(t)\nabla t +
    \int_{a}^{b}g(t)\nabla t$;
$\int_{a}^{b} \alpha f(t)\nabla t =\alpha
    \int_{a}^{b}f(t)\nabla t$;
$\int_{a}^{b}  f(t)\nabla t = - \int_{b}^{a} f(t)\nabla t$;
$\int_{a}^{a}  f(t)\nabla t=0$;
$\int_{a}^{b}  f(t)\nabla t = \int_{a}^{c}  f(t)\nabla t + \int_{c}^{b} f(t)\nabla t$.
If $f(t)> 0$ for all $a < t \leq b$, then
$\int_{a}^{b}  f(t)\nabla t > 0$;
$\int_{a}^{b}f^\rho(t)g^{\nabla}(t)\nabla t
=\left[(fg)(t)\right]_{t=a}^{t=b}-\int_{a}^{b}f^{\nabla}(t)g(t)\nabla t$;
$\int_{a}^{b}f(t)g^{\nabla}(t)\nabla t
=\left[(fg)(t)\right]_{t=a}^{t=b}-\int_{a}^{b}f^{\nabla}(t)g^\rho(t)\nabla t$.
\end{Th}

\begin{Remark}
Let $a, b \in \mathbb{T}$ and $f\in C_{ld}(\mathbb{T}, \mathbb{R})$.
For $\mathbb{T}=\mathbb{R}$, then $\int_{a}^{b}f(t)\nabla t =
\int_{a}^{b}f(t) dt$, where the integral on the right side is the
usual Riemann integral. For $\mathbb{T}=\mathbb{Z}$, then
$\displaystyle \int_{a}^{b}f(t)\nabla t = \sum_{t=a+1}^{b}f(t)$ if
$a<b$, $\displaystyle \int_{a}^{b}f(t)\nabla t=0$ if $a=b$, and
$\displaystyle \int_{a}^{b}f(t)\nabla t = - \sum_{t=b+1}^{a}f(t)$ if
$a>b$.
\end{Remark}

Let $a, b \in \mathbb{T}$ with $a<b$. We define the interval $[a,b]$
in $\mathbb{T}$ by
$[a,b]:=\{t \in \mathbb{T}: a \leq t \leq b\}$.
Open intervals and half-open intervals in $\mathbb{T}$ are defined
accordingly. Note that $[a,b]_\kappa=[a,b]$ if $a$ is right-dense
and $[a,b]_\kappa=[\sigma(a),b]$ if $a$ is right-scattered.

\begin{Lemma}{\rm (\cite{MT})}
\label{lemma4}
Let $f, g \in C_{ld}([a,b], \mathbb{R})$. If
$\int_{a}^{b}\left(f(t)\eta^\rho(t) + g(t)
\eta^{\nabla}(t)\right)\nabla t=0$
for all  $\eta \in C^1_{ld}([a,b], \mathbb{R})$ such that
$\eta(a)=\eta(b)=0$, then $g$  is nabla differentiable and
$g^{\nabla}(t)=f(t) \ \ \ \forall t\in [a,b]_\kappa$.
\end{Lemma}


\section{Main Results}
\label{sec:mr}

Throughout we let $A,B\in \T$ with $A<B$.
Now let $[a,b]$ be a subinterval of $[A,B]$,
with $a,b\in \T$ and $A<a$. The problem of
the calculus of variations on time scales under our consideration
consists of minimizing or maximizing
\begin{equation}\label{vp}
\mathcal{L}[x]=\int_{a}^{b}f(t,x^{\rho}(t),x^{\nabla}(t),x(a),
x(b))\nabla t , \quad (x(a)=x_{a}), \quad (x(b)=x_{b})
\end{equation}
over all $x\in C^{1}_{ld}([A,b],\R)$. Using parentheses around the
endpoint conditions means that the conditions may or may not be
present. We assume that $f(t,x,v,z,s):[A,b]\times \R^{4} \rightarrow
\R$ has partial continuous derivatives with respect to $x,v,z,s$ for
all $t\in[A,b]$, and $f(t,\cdot,\cdot,\cdot,\cdot)$ and its partial
derivatives are ld-continuous for all $t\in[A,b]$.

A function $x\in C^{1}_{ld}([A,b],\R)$ is said to be an admissible
function provided that it satisfies the endpoints conditions (if
any is given).
Let us consider the following norm in $C^{1}_{ld}([A,b],\R)$:
$\|x\|_{1}=\sup_{t\in[A,b]}\left|x^{\rho}(t)\right|
   +\sup_{t\in[A,b]}\left|x^{\nabla}(t)\right|$.

\begin{df}
An admissible function $\tilde{x}$ is said to be a \emph{weak local
minimizer} (respectively \emph{weak local maximizer}) for \eqref{vp}
if there exists $\delta
>0$ such that $\mathcal{L}[\tilde{x}]\leq \mathcal{L}[x]$
(respectively $\mathcal{L}[\tilde{x}] \geq \mathcal{L}[x]$)
for all admissible $x$ with $\|x-\tilde{x}\|_{1}<\delta$.
\end{df}


\subsection{Lagrangian approach}
\label{sec:mr:lag}

Next theorem gives necessary optimality conditions for the problem
\eqref{vp}.

\begin{Th}
\label{main} If $\tilde{x}$ is an extremizer
(\textrm{i.e.}, a weak local
minimizer or a weak local maximizer) for the problem \eqref{vp},
then
\begin{equation}\label{Euler}
f_{x^{\nabla}}^{\nabla}(t,\tilde{x}^{\rho}(t),\tilde{x}^{\nabla}(t),\tilde{x}(a),\tilde{x}(b))=
f_{x^{\rho}}(t,\tilde{x}^{\rho}(t),\tilde{x}^{\nabla}(t),\tilde{x}(a),\tilde{x}(b))
\end{equation}
for all $t \in [a,b]_{\kappa}$. Moreover, if $x(a)$ is not
specified, then
\begin{equation}\label{new:bcb}
f_{x^{\nabla}}(a,\tilde{x}^{\rho}(a),\tilde{x}^{\nabla}(a),\tilde{x}(a),\tilde{x}(b))
=\int_{a}^{b}f_{z}(t,\tilde{x}^{\rho}(t),\tilde{x}^{\nabla}(t),
\tilde{x}(a),\tilde{x}(b))\nabla t;
\end{equation}
if $x(b)$ is not specified, then
\begin{equation}\label{new:bce}
f_{x^{\nabla}}(b,\tilde{x}^{\rho}(b),\tilde{x}^{\nabla}(b),\tilde{x}(a),\tilde{x}(b))
=-\int_{a}^{b}f_{s}(t,\tilde{x}^{\rho}(t),\tilde{x}^{\nabla}(t),
\tilde{x}(a),\tilde{x}(b))\nabla t.
\end{equation}
\end{Th}

\begin{proof}
Suppose that $L$ has a weak local extremum at $\tilde{x}$. We can
proceed as Lagrange did, by considering the value of $L$ at a nearby
function $x = \tilde{x} + \varepsilon h$, where $\varepsilon\in \R$
is a small parameter, $h \in C^{1}_{ld}([A,b],\R)$. We do not
require $h(a)=0$ or $h(b)=0$ in case $x(a)$ or $x(b)$, respectively,
is free (it is possible that both are free). Let
\begin{equation*}
\begin{split}
\phi(\varepsilon)
&= L[(\tilde{x} + \varepsilon h)(\cdot)]\\
&= \int_a^b f(t,\tilde{x}^{\rho}(t)+\varepsilon
h(t),\tilde{x}^{\nabla}(t)+ \varepsilon
h^{\nabla}(t),\tilde{x}(a)+\varepsilon h(a),\tilde{x}(b)
+\varepsilon h(b)) \nabla t.
\end{split}
\end{equation*}
A necessary condition for $\tilde{x}$ to be an extremizer is given
by
\begin{multline}
\label{eq:FT} \left.\phi'(\varepsilon)\right|_{\varepsilon=0} = 0\\
\Leftrightarrow \int_a^b \Bigl[ f_{x^{\rho}}(\cdots) h^{\rho}(t) +
f_{x^{\nabla}}(\cdots) h^{\nabla}(t) + f_z(\cdots) h(a) +f_s(\cdots)
h(b)\Bigr]\triangle t = 0 \, ,
\end{multline}
where $(\cdots) =
\left(t,\tilde{x}^{\rho}(t),\tilde{x}^{\nabla}(t),\tilde{x}(a),
\tilde{x}(b)\right)$. Integration by parts gives
\begin{multline}
\label{eq:aft:IP}
0= \int_a^b
\left(f_{x^{\rho}}(\cdots)-f_{x^{\nabla}}^{\nabla}(\cdots)\right)h^{\rho}(t)
\nabla t
+h(b)\left(f_{x^{\nabla}}(\cdots)|_{t=b}+\int_a^b
f_{s}(\cdots)\nabla t \right)\\
+h(a)\left(-f_{x^{\nabla}}(\cdots)|_{t=a}+\int_a^b f_{z}(\cdots)\nabla t \right).
\end{multline}
We first consider functions $h(t)$ such that $h(a) =h(b)=0$. Then,
by Lemma~\ref{lemma4}, we have
\begin{equation}
\label{eq:EL} f_{x^{\rho}}(\cdots)-f_{x^{\nabla}}^{\nabla}(\cdots)=0
\end{equation}
for all $t\in[a,b]_{\kappa}$. Therefore, in order for $\tilde{x}$ to
be an extremizer for the problem \eqref{vp}, $\tilde{x}$ must be a
solution of the nabla differential Euler-Lagrange equation. But if
$\tilde{x}$ is a solution of \eqref{eq:EL}, the first integral in
expression \eqref{eq:aft:IP} vanishes, and then the condition
\eqref{eq:FT} takes the form
\begin{multline*}
h(b)\left(f_{x^{\nabla}}(\cdots)|_{t=b}+\int_a^b f_{s}(\cdots)\nabla t \right) \\
+h(a)\left(-f_{x^{\nabla}}(\cdots)|_{t=a}+\int_a^b
f_{z}(\cdots)\nabla t \right)=0.
\end{multline*}
If $x(a)=x_{a}$ and $x(b)=x_{b}$ are given in the formulation of
problem \eqref{vp}, then the  latter equation is trivially satisfied
since $h(a)=h(b)=0$. When $x(a)$ is free, then \eqref{new:bcb} holds;
when $x(b)$ is free, then \eqref{new:bce} holds;
since $h(a)$ or $h(b)$ is, respectively, arbitrary.
\end{proof}

Letting $\T=\R$ in Theorem~\ref{main} we immediately obtain the
corresponding result in the classical context of the calculus of
variations.

\begin{Cor}{\rm (cf. \cite{withAlanZinober,Basia:post_doc_Aveiro:2})}
\label{natural:R}
Let $\mathbb{T}=\R$. If $\tilde{x}$ is an extremizer for
\begin{equation*}
\mathcal{L}[x]=\int_{a}^{b}f(t,x(t),x'(t),x(a), x(b))dt , \quad
(x(a)=x_{a}), \quad (x(b)=x_{b}),
\end{equation*}
then
\begin{equation*}
\frac{d}{dt}f_{x'}(t,\tilde{x}(t),\tilde{x}'(t),\tilde{x}(a),\tilde{x}(b))=
f_{x}(t,\tilde{x}(t),\tilde{x}'(t),\tilde{x}(a),\tilde{x}(b))
\end{equation*}
for all $t \in [a,b]$. Moreover, if $x(a)$ is free, then
\begin{equation}\label{new:bc:Rb}
f_{x'}(a,\tilde{x}(a),\tilde{x}'(a),\tilde{x}(a),\tilde{x}(b))=\int_{a}^{b}f_{z}(t,\tilde{x}(t),\tilde{x}'(t),
\tilde{x}(a),\tilde{x}(b)) dt;
\end{equation}
if $x(b)$ is free, then
\begin{equation}\label{new:bc:Re}
f_{x'}(b,\tilde{x}(b),\tilde{x}'(b),\tilde{x}(a),\tilde{x}(b))=-\int_{a}^{b}f_{s}(t,\tilde{x}(t),\tilde{x}'(t),
\tilde{x}(a),\tilde{x}(b))dt.
\end{equation}
\end{Cor}

\begin{Example}
Consider a river with parallel straight banks, $b$ units apart.
One of the banks coincides with the $y$ axis, the water is assumed to be moving
parallel to the banks with speed $v$ that depends, as usual, on the $x$ coordinate,
but also on the arrival point $y(b)$ ($y(b)$ is not given and is part of the solution
of the problem). A boat with constant speed $c$ ($c^2>v^2$)
in still water is crossing the river in the short possible time,
using the point $y(0) =0$ as point of departure. The endpoint $y(b)$ is allowed to move freely
along the other bank $x=b$. Then one can easily obtain that the time of passage along the path $y(x)$ is given by
\begin{equation*}
\mathcal{T}[y]=\int_{0}^{b}\frac{\sqrt{c^2(1+(y'(x))^2)-v^2(x,y(b))}-v(x,y(b))y'(x)}{c^2-v^2(x,y(b))}dx,
\end{equation*}
where $v=v(x,y(b))$ is a known function of $x$ and $y(b)$. This is not a standard problem
because the integrand depends on $y(b)$. Corollary~\ref{natural:R} gives the solution.
\end{Example}

\begin{Remark}
In the classical setting $f$ does not depend on $x(a)$ and $x(b)$, \textrm{i.e.}, $f_z = 0$ and $f_s = 0$. In that case
(\ref{new:bc:Rb}) and
(\ref{new:bc:Re}) reduce to the well known natural boundary
conditions $f_{x'}\left(a,\tilde{x}(a),\tilde{x}'(a)\right) = 0$ and
$f_{x'}\left(b,\tilde{x}(b),\tilde{x}'(b)\right) = 0$.
\end{Remark}

Similarly, we can obtain other corollaries by choosing different
time scales. The next corollary is obtained from Theorem~\ref{main}
letting $\T=\Z$.

\begin{Cor}
If $\tilde{x}$ is an extremizer for
\begin{equation*}
 L[x] = \sum _{t=a+1}^{b}f(t,x(t-1),\nabla x(t),x(a),x(b)), \quad (x(a)=x_{a}), \quad (x(b)=x_{b}),
\end{equation*}
then $f_{x}\left(t,\tilde{x}(t-1),\nabla \tilde{x}(\tau),\tilde{x}(a),\tilde{x}(b)\right)=
\nabla f_{v}\left(t,\tilde{x}(t-1),\nabla \tilde{x}(t),\tilde{x}(a),\tilde{x}(b)\right)$
for all $t \in [a+1,b]$. Moreover,
\begin{equation*}
f_{v}(a,\tilde{x}(a-1),\nabla
\tilde{x}(a),\tilde{x}(a),\tilde{x}(b)) =\sum
_{t=a+1}^{b}f_{z}(t,\tilde{x}(t-1),\nabla
\tilde{x}(t),\tilde{x}(a),\tilde{x}(b)),
\end{equation*}
if $x(a)$ is not specified and
\begin{equation*}
f_{v}(b,\tilde{x}(b-1),\nabla
\tilde{x}(b),\tilde{x}(a),\tilde{x}(b)) =-\sum
_{t=a+1}^{b}f_{s}(t,\tilde{x}(t-1),\nabla
\tilde{x}(t),\tilde{x}(a),\tilde{x}(b)),
\end{equation*}
if $x(b)$ is not specified.
\end{Cor}

Let $\mathbb{T}=q^{\mathbb{N}_{0}}$, $q>1$. To simplify notation, we
use $\nabla_{q}$ for the $q$-nabla derivative:
$\nabla_{q}x(t)=\frac{x(t)-x(tq^{-1})}{t(1-q^{-1})}$.

\begin{Cor}
If $\tilde{x}$ is an extremizer for
\begin{gather*}
L[x] = (1-q^{-1})\sum _{t\in(a,b]}t f
\left(t,x(q^{-1}t),\nabla_{q}x(t),x(a),x(b)\right),\\
(x(a)=x_{a}), \quad (x(b)=x_{b}),
\end{gather*}
then
$f_{x}\left(t,\tilde{x}(q^{-1}t),\nabla_{q}\tilde{x}(t),\tilde{x}(a),\tilde{x}(b)\right)
= \nabla_{q}f_{
v}\left(t,\tilde{x}(q^{-1}t),\nabla_{q}\tilde{x}(t),\tilde{x}(a),\tilde{x}(b)\right)$
for all $t \in (a,b]$. Moreover, if $x(a)$ is free, then
\begin{multline*}
f_{v}\left(a,\tilde{x}(aq^{-1}),\nabla_{q}\tilde{x}(a),\tilde{x}(a),\tilde{x}(b)\right)\\
=(1-q^{-1})\sum _{t\in(a,b]}t
f_{z}\left(t,\tilde{x}(q^{-1}t),\nabla_{q}\tilde{x}(t),\tilde{x}(a),\tilde{x}(b)\right);
\end{multline*}
if $x(b)$ is free, then
\begin{multline*}
f_{v}\left(b,\tilde{x}(bq^{-1}),\nabla_{q}\tilde{x}(b),\tilde{x}(a),\tilde{x}(b)\right)\\
=-(1-q^{-1})\sum _{t\in(a,b]}t
f_{s}\left(t,\tilde{x}(q^{-1}t),\nabla_{q}\tilde{x}(t),\tilde{x}(a),\tilde{x}(b)\right).
\end{multline*}
\end{Cor}

We illustrate the application of Theorem~\ref{main} with an
example.

\begin{Example}
\label{lagrangform}
Consider the problem
\begin{equation}\label{ex:1}
 \text{minimize} \quad   \mathcal{L}[x]=\int_{0}^{1}\left((x^{\nabla}(t))^{2}+\alpha x^2(0)+\beta(x(1)-1)^2\right)\nabla t
\end{equation}
where $\alpha,\beta \in \R^{+}$. If $\tilde{x}$ is a local minimizer
of \eqref{ex:1}, then conditions \eqref{Euler}--\eqref{new:bce} must hold, \textrm{i.e.},
\begin{equation}\label{ex:1:1}
(2\tilde{x}^{\nabla}(t))^{\nabla}=0,
\end{equation}
\begin{equation}
\label{ex:1:2}
2\tilde{x}^{\nabla}(0)=\int_{0}^{1}2\alpha x(0)\nabla t,
\quad 2\tilde{x}^{\nabla}(1)=-\int_{0}^{1}2\beta (x(1)-1)\nabla t.
\end{equation}
Equation \eqref{ex:1:1} implies that there exists a constant $c\in
\R$ such that $\tilde{x}^{\nabla}(t)=c$.
Solving this equation we obtain $\tilde{x}(t)=ct+\tilde{x}(0)$. In
order to determine $c$ and $\tilde{x}(0)$ we use the natural
boundary conditions \eqref{ex:1:2} which we can now rewrite as a
system of two equations:
\begin{equation}
\label{ex:1:3}
c-\alpha \tilde{x}(0)=0, \quad
c+\beta (c+\tilde{x}(0)-1)=0.
\end{equation}
The solution of \eqref{ex:1:3} is $c=\frac{\alpha \beta}{\alpha
+\beta+\alpha \beta}$ and $\tilde{x}(0)=\frac{\beta}{\alpha
+\beta+\alpha \beta}$. Hence, $\tilde{x}(t)=c(\alpha, \beta)
t+\tilde{x}(0,\alpha, \beta)$ is a candidate for minimizer. We note
that
$\lim_{\alpha, \beta \rightarrow \infty}c( \alpha, \beta)= 1$,
$\lim_{\alpha, \beta \rightarrow \infty}\tilde{x}(0,\alpha,
\beta)= 0$, and in the limit $\alpha, \beta \rightarrow \infty$ the solution of
\eqref{ex:1} coincides with the solution of the following problem
with fixed initial and terminal points:
$\min \mathcal{L}[x]
=\int_{0}^{1} (x^{\nabla}(t))^{2} \nabla t$,
subject to $x(0) = 0$ and $x(1) = 1$.
\begin{figure}[t]
\begin{center}
\psfrag{alpha=beta=2}{{\tiny $\alpha=\beta=2$}}
\psfrag{alpha=beta=4}{{\tiny $\alpha=\beta=4$}}
\psfrag{alpha=beta=20}{{\tiny $\alpha=\beta=20$}}
\psfrag{beta=inf}{{\tiny $\beta=\infty$}}
\includegraphics[width=8cm,height=6cm,angle=-90]{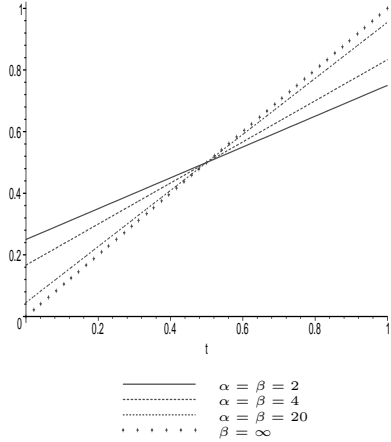}
\vspace*{-1.2cm}
\caption{The extremal $\tilde{x}(t)=c(\alpha, \beta)
t+\tilde{x}(0,\alpha, \beta)$ of Example~\ref{lagrangform} for
different values of parameters $\alpha$ and $\beta$.}
\end{center}
\end{figure}
Expression $\alpha x^2(0)+\beta(x(1)-1)^2$ added to the Lagrangian
$(x^{\nabla}(t))^{2}$ works like a penalty function
when $\alpha$ and $\beta$ go to infinity.
The penalty function itself grows, and forces the merit function \eqref{ex:1}
to increase in value when the constraints $x(0) = 0$ and $x(1) = 1$
are violated, and causes no growth when constraints are fulfilled.
\end{Example}


\subsection{Hamiltonian approach}
\label{sec:mr:ham}

Now let us consider the more general variational problem of optimal control on
time scales: to minimize (maximize) the functional
\begin{equation}\label{ocp}
\mathcal{L}[x,u]=\int_{a}^{b}f(t,x^{\rho}(t),u^{\rho}(t),x(a),x(b))\nabla
     t ,
\end{equation}
subject to
\begin{equation}\label{ocbc}
\begin{split}
x^{\nabla}(t)=g(t,x^{\rho}(t),u^{\rho}(t),x(a),x(b)),\\
(x(a)=x_{a}), \quad (x(b)=x_{b}),
\end{split}
\end{equation}
where $x_{a},x_{b}\in \R$, $f(t,x,v,z,s):[A,b]\times
\R^{4}\rightarrow \R$ and $g(t,x,v,z,s):[A,b]\times \R^{4}
\rightarrow \R$ have partial continuous derivatives with respect to
$x,v,z,s$ for all $t\in[A,b]$, and $f(t,\cdot,\cdot,\cdot,\cdot)$,
$g(t,\cdot,\cdot,\cdot,\cdot)$ and
their partial derivatives are ld-continuous for all $t$.
We also assume that the function $g_{x}$ is $\nu$-regressive.

A necessary optimality condition for problem
\eqref{ocp}--\eqref{ocbc} can be obtained from a general Lagrange
multiplier theorem in space of infinite dimension. We form a
Lagrange function $f+\lambda^{\rho}(g-x^{\nabla})$ by introducing a
multiplier $\lambda:[A,b]\rightarrow \R$. In what follows we shall
assume that $\lambda^{\rho}$ is a nabla differentiable function on
$[a,b]$. For examples of time scales for which the composition of a
nabla differentiable function with $\rho$ is not nabla
differentiable, we refer the reader to \cite{BP1}. Note that we are interested in the study of normal extremizers only.
In general one needs to replace $f$ in $f+\lambda^{\rho}(g-x^{\nabla})$ by $\lambda_{0}f$.
Normal extremizers correspond to $\lambda_{0}=1$
while abnormal ones correspond to $\lambda_{0}=0$.

\begin{Th}
\label{hamiltonian} If $(\tilde{x},\tilde{u})$ is a normal
extremizer for the problem \eqref{ocp}--\eqref{ocbc}, then there
exists a function $\tilde{p}$ such that the triple
$(\tilde{x},\tilde{u},\tilde{p})$ satisfies the Hamiltonian system
\begin{equation}
\label{ham1}
x^{\nabla}(t)=H_{p}(t,x^{\rho}(t),u^{\rho}(t),p(t),x(a),x(b)),
\end{equation}
\begin{equation}\label{ham2}
(p(t))^{\nabla}=-H_{x^{\rho}}(t,x^{\rho}(t),u^{\rho}(t),p(t),x(a),x(b)),
\end{equation}
the stationary condition
\begin{equation}\label{ham3}
H_{u^{\rho}}(t,x^{\rho}(t),u^{\rho}(t),p(t),x(a),x(b))=0,
\end{equation}
for all $t \in [a,b]_{\kappa}$, and the transversality condition
\begin{equation}\label{ham4}
p(a)=-\int_{a}^{b}H_{z}(t,x^{\rho}(t),u^{\rho}(t),p(t),x(a),x(b))\nabla
t,
\end{equation}
when $x(a)$ is free; the transversality condition
\begin{equation}\label{ham5}
p(b)=\int_{a}^{b}H_{s}(t,x^{\rho}(t),u^{\rho}(t),p(t),x(a),x(b))\nabla
t,
\end{equation}
when $x(b)$ is free, where the Hamiltonian
$H(t,x,v,p,z,s):[A,b]\times \R^{5} \rightarrow \R$ is defined by
\begin{equation*}
H(t,x^{\rho},u^{\rho},p,x(a),x(b))=f(t,x^{\rho},u^{\rho},x(a),x(b))
+pg(t,x^{\rho},u^{\rho},x(a),x(b)).
\end{equation*}
\end{Th}
\begin{proof}
Let $(\tilde{x},\tilde{u})$ be a normal extremizer for the problem
\eqref{ocp}--\eqref{ocbc}. Using the Lagrange multiplier rule, we
form the expression $\lambda^{\rho}(g-x^{\nabla})$ for each value of
$t$ (we are assuming that $\mathbb{T}$ is a time scale for which
$\lambda^{\rho}$ is a nabla differentiable function on $[a,b]$). The
replacement of $f$ by $f+\lambda^{\rho}(g-x^{\nabla})$ in the
objective functional gives us a new problem: minimize (maximize)
\begin{equation}
\label{new:ocp}
\begin{split}
\mathcal{I}[x,u,\lambda]=\int_{a}^{b}\Bigl\{&f(t,x^{\rho}(t),u^{\rho}(t),x(a),x(b))\\
&\quad +\lambda^{\rho}(t)[g(t,x^{\rho}(t),u^{\rho}(t),x(a),x(b))
-x^{\nabla}(t)]\Bigr\}\nabla t, \\
&(x(a)=x_{a}), \quad (x(b)=x_{b}).
\end{split}
\end{equation}
Substituting
\begin{equation*}
H(t,x^{\rho},u^{\rho},\lambda^{\rho},x(a),x(b))
=f(t,x^{\rho},u^{\rho},x(a),x(b))
+\lambda^{\rho}g(t,x^{\rho},u^{\rho},x(a),x(b))
\end{equation*}
into \eqref{new:ocp} we can simplify the new functional to the form
\begin{equation}\label{ham:ocp}
\mathcal{I}[x,u,\lambda]=\int_{a}^{b}[H(t,x^{\rho},u^{\rho},\lambda^{\rho},x(a),x(b))
-\lambda^{\rho}(t)x^{\nabla}(t)]\nabla t .
\end{equation}
The choice of $\lambda^{\rho}$ will produce no effect on the value
of the functional $\mathcal{I}$, as long as the equation
$x^{\nabla}(t)=g(t,x^{\rho}(t),u^{\rho}(t),x(a),x(b))$ is satisfied,
i.e., as long as
\begin{equation}\label{s1}
x^{\nabla}(t)=H_{\lambda^{\rho}}(t,x^{\rho}(t),u^{\rho}(t),
\lambda^{\rho}(t),x(a),x(b)).
\end{equation}
Therefore, we impose \eqref{s1} as a necessary condition for the
minimizing (maximizing) of the functional $\mathcal{I}$. Under
condition \eqref{s1} the free extremum of the $\mathcal{I}$ is
identical with the constrained extremum of the functional
$\mathcal{L}$. In view of \eqref{ham:ocp}, applying
Theorem~\ref{main} to the problem \eqref{new:ocp} gives
\begin{equation}\label{s2}
(\lambda^{\rho}(t))^{\nabla}=-H_{x^{\rho}}(t,x^{\rho}(t),
u^{\rho}(t),\lambda^{\rho}(t),x(a),x(b)),
\end{equation}
\begin{equation}\label{s3}
H_{u^{\rho}}(t,x^{\rho}(t),u^{\rho}(t),\lambda^{\rho}(t),
x(a),x(b))=0,
\end{equation}
for all $t \in [a,b]_{\kappa}$, and the transversality conditions
\begin{equation}\label{trans}
\begin{split}
\lambda^{\rho}(a)&=-\int_{a}^{b}H_{z}(t,x^{\rho}(t),u^{\rho}(t),
\lambda^{\rho}(t),x(a),x(b))\nabla t,\\
\lambda^{\rho}(b)&=\int_{a}^{b}H_{s}(t,x^{\rho}(t),u^{\rho}(t),
\lambda^{\rho}(t),x(a),x(b))\nabla t,
\end{split}
\end{equation}
in case $x(a)$ and $x(b)$ are free. Note that \eqref{s2} is the
first order nonhomogeneous linear equation and from the assumptions
on $f$ and $g_{x}$, the solution $\tilde{\lambda}^{\rho}$ exists
(see Theorem~3.42 in \cite{BP2}). Therefore the triple
$(\tilde{x},\tilde{u},\tilde{\lambda}^{\rho})$ satisfies the system
\eqref{s1}--\eqref{s3} and the transversality conditions
\eqref{trans} in case $x(a)$ and $x(b)$ are free. Putting
$\tilde{p}=\tilde{\lambda}^{\rho}$ we obtain the intended conditions
\eqref{ham1}--\eqref{ham5}.
\end{proof}

\begin{Remark}
Theorem~\ref{hamiltonian} covers the case when
$(\tilde{x},\tilde{u})$ is a normal extremizer for the problem
\eqref{ocp}--\eqref{ocbc}. We do not consider problems with abnormal
extremizers, but in general such extremizers are possible. Let us
consider the problem
\begin{equation}
\begin{gathered}
\text{minimize} \quad \mathcal{L}[x,u]=\int_{0}^{1}(u(t))^2dt ,\\
x'(t)=0,\\
x(0)=0, \quad x(1)=0
 \end{gathered}
\end{equation}
defined on $\T=\R$. Then, the pair $(\tilde{x}(t),\tilde{u}(t))=(0,0)$ is abnormal
minimizer for this problem. Observe that
$\mathcal{I}[\tilde{x}(t),\tilde{u}(t),\lambda]=0$ for all
$\lambda\in C^{1}([0,1],\R)$. However, for the triple
$(x(t),u(t),\lambda(t))=(t^2 -t,0,2t-1)$ we have
$\mathcal{I}[x(t),u(t),\lambda(t)]=\int_{0}^{1}-(2t-1)^2
dt=-\frac{1}{3}<0$.
\end{Remark}

\begin{Example}\label{ham:nc}
Consider the problem
\begin{equation}\label{ham:nc:p}
\begin{gathered}
\text{minimize} \quad
\mathcal{L}[x,u]=\int_{0}^{3}(u^{\rho}(t))^2+t^{2}(x(3)-1)^2+t^{2}(x(0)-1)^2
\nabla t ,\\
x^{\nabla}(t)=u^{\rho}(t).
 \end{gathered}
\end{equation}
To find candidate solutions for the problem, we start by forming the
Hamiltonian function
\begin{equation*}
H(t,x^{\rho},u^{\rho},p,x(0),x(3))=(u^{\rho})^2+t^{2}(x(3)-1)^2
+t^{2}(x(0)-1)^2+pu^{\rho}.
\end{equation*}
Candidate solutions $(\tilde{x},\tilde{u})$ are those satisfying
the following conditions:
\begin{equation}
\label{ham:nc:h1}
(p(t))^{\nabla}=0, \quad
u^{\rho}(t)=x^{\nabla}(t), \quad
2u^{\rho}(t)+p(t)=0,
\end{equation}
\begin{equation}
\label{ham:nc:4}
p(0)=-\int_{0}^{3}2t^{2}(x(0)-1)\nabla t,\quad
p(3)=\int_{0}^{3}2t^{2}(x(3)-1)\nabla t.
\end{equation}
From \eqref{ham:nc:h1} we conclude that $p(t)=c$
and a possible solution is $\tilde{x}(t)=-\frac{c}{2}t+d$, where
$c$, $d$ are constants of nabla integration. In order to determine
$c$ and $d$ we use the transversality conditions \eqref{ham:nc:4}
that we can write as
\begin{equation}
\label{ham:nc:h5}
c =-\int_{0}^{3}2t^{2}(d-1)\nabla t,\quad
c =\int_{0}^{3}2t^{2}\left(-\frac{3c}{2}+d-1\right)\nabla t.
\end{equation}
The values of the nabla integrals in \eqref{ham:nc:h5} depend on the
time scale. Notwithstanding this fact, substituting
$\int_{0}^{3}t^{2}\nabla t=k$, $k\in \R$, into \eqref{ham:nc:h5} we
can simplify equations to the form
\begin{equation}
\label{ham:nc:di}
c=-2k(d-1),\quad
c=2k\left(-\frac{3c}{2}+d-1 \right).
\end{equation}
Equations \eqref{ham:nc:di} yield $c=0$ and $d=1$. Therefore, the
extremal of the problem \eqref{ham:nc:p} is $\tilde{x}(t)=1$ on any time scale.
\end{Example}

When $\T=\R$ we obtain from Theorem~\ref{hamiltonian}
the following corollary.

\begin{Cor}\label{forR}
Let $(\tilde{x},\tilde{u})$ be a normal extremizer for
\begin{equation*}
L[x,u] = \int_a^b f(t,x(t),u(t),x(a),x(b)) dt
\end{equation*}
subject to
\begin{equation*}
\begin{gathered}
x'(t)=g(t,x(t),u(t),x(a),x(b))\\
(x(a) = x_{a}) \quad (x(b)=x_{b}),
\end{gathered}
\end{equation*}
where $a,b\in \R$, $a<b$. Then there exists a function $\tilde{p}$
such that the triple $(\tilde{x},\tilde{u},\tilde{p})$ satisfies the
Hamiltonian system
\begin{equation*}
x'(t)=H_{\lambda}, \quad p'(t)=-H_{x},
\end{equation*}
the stationary condition
\begin{equation*}
H_{u}=0,
\end{equation*}
for all $t \in [a,b]$ and the transversality condition
\begin{equation*}
p(a) =-\int_{a}^{b}H_{z}dt,
\end{equation*}
when $x(a)$ is free; the transversality condition
\begin{equation*}
p(b) =\int_{a}^{b}H_{s}dt,
\end{equation*}
when $x(b)$ is free, where the Hamiltonian $H$ is defined by
$$H(t,x,u,p,z,s)=f(t,x,u,z,s) + p \, g(t,x,u,z,s).$$
\end{Cor}

We illustrate the use of Corollary~\ref{forR} with an example.

\begin{Example}\label{ham:nc:R}
Consider the problem
\begin{equation}\label{ham:nc:rp}
\begin{gathered}
\text{minimize} \quad \mathcal{L}[x,u]=\int_{-1}^{1}(u(t))^2dt ,\\
x'(t)=u(t)+x(-1)t+x(1)t.
 \end{gathered}
\end{equation}
We begin by writing the Hamiltonian function
\begin{equation*}
H(t,x,u,p,x(-1),x(1))=u^2+p(u+x(-1)t+x(1)t).
\end{equation*}
Candidate solutions $(\tilde{x},\tilde{u})$ are those satisfying the
following conditions:
\begin{equation}\label{ham:nc:rh1}
p'(t)=0,
\end{equation}
\begin{equation}\label{ham:nc:rh2}
x'(t)=u(t)+x(-1)t+x(1)t,
\end{equation}
\begin{equation}\label{ham:nc:rh3}
2u(t)+p(t)=0,
\end{equation}
\begin{equation}\label{ham:nc:rh4}
\begin{split}
p(-1)&=-\int_{-1}^{1}p(t)tdt, \\
p(1)&=\int_{-1}^{1}p(t)tdt.
\end{split}
\end{equation}
The equation \eqref{ham:nc:rh1} has solution $\tilde{p}(t)=c$,
$-1\leq t \leq 1$, which upon substitution into \eqref{ham:nc:rh4}
yields
\begin{equation*}
c=\int_{-1}^{1}ctdt=0.
\end{equation*}
From the stationary condition \eqref{ham:nc:rh3} we get
$\tilde{u}(t)=0$. Therefore, $\mathcal{L}[\tilde{x},\tilde{u}]=0$.
Finally, substituting the optimal control candidate back into
\eqref{ham:nc:rh2} yields
\begin{equation}\label{ham:nc:rh3:d}
\tilde{x}'(t)=\tilde{x}(-1)t+\tilde{x}(1)t.
\end{equation}
Integrating equation \eqref{ham:nc:rh3:d} we obtain
\begin{equation}\label{ham:nc:r5}
\tilde{x}(t)=\frac{1}{2}t^{2}(\tilde{x}(-1)+\tilde{x}(1))+d.
\end{equation}
Substituting $t=1$ and $t=-1$ into \eqref{ham:nc:r5} we get $d=0$
and $\tilde{x}(-1)=\tilde{x}(1)$. Therefore, extremals of the
problem \eqref{ham:nc:rp} are $\tilde{x}(t)=t^{2}\tilde{x}(1)$,
where $\tilde{x}(1)$ is any real number.
\end{Example}

\begin{Th}\label{sc}
Let $(x^{\rho},u^{\rho},z,s)\rightarrow f(t,x^{\rho},u^{\rho},z,s)$
and $(x^{\rho},u^{\rho},z,s)\rightarrow g(t,x^{\rho},u^{\rho},z,s)$
be jointly convex (concave) in $(x^{\rho},u^{\rho},z,s)$ for any
$t$. If $(\tilde{x},\tilde{u},\tilde{p})$ is a solution of system
\eqref{ham1}--\eqref{ham5} and $\tilde{p}(t)\geq 0$ for all
$t\in[a,b]$, then $(\tilde{x},\tilde{u})$ is a global minimizer
(maximizer) of problem \eqref{ocp}--\eqref{ocbc}.
\end{Th}

\begin{proof}
We shall give the proof for the convex case. Since $f$ is jointly
convex in $(x^{\rho},u^{\rho},z,s)$ for any admissible pair $(x,u)$,
we have
\begin{equation*}
\begin{split}
\mathcal{L}[x,u]&-\mathcal{L}[\tilde{x},\tilde{u}]\\
&= \int_{a}^{b}\left[f(t,x^{\rho}(t),u^{\rho}(t),x(a),x(b))
-f(t,\tilde{x}^{\rho}(t),\tilde{u}^{\rho}(t),\tilde{x}(a),
\tilde{x}(b))\right]\nabla t\\
&\geq \int_{a}^{b}\Bigl[f_{x^{\rho}}(t,\tilde{x}^{\rho}(t),\tilde{u}^{\rho}(t),\tilde{x}(a),\tilde{x}(b))(x^{\rho}(t)
-\tilde{x}^{\rho}(t))\\
&\qquad\qquad  +f_{u^{\rho}}(t,\tilde{x}^{\rho}(t),\tilde{u}^{\rho}(t),\tilde{x}(a),\tilde{x}(b))(u^{\rho}(t)
-\tilde{u}^{\rho}(t))\\
&\qquad \qquad +f_{z}(t,\tilde{x}^{\rho}(t),\tilde{u}^{\rho}(t),\tilde{x}(a),\tilde{x}(b))(x(a)-\tilde{x}(a))\\
&\qquad \qquad +f_{s}(t,\tilde{x}^{\rho}(t),\tilde{u}^{\rho}(t),\tilde{x}(a),\tilde{x}(b))(x(b)-\tilde{x}(b))\Bigr]\nabla t.
\end{split}
\end{equation*}
Because the triple $(\tilde{x},\tilde{u},\tilde{p})$ satisfies
equations \eqref{ham2}--\eqref{ham5}, we obtain
\begin{equation*}
\begin{split}
\mathcal{L}[x,u] &-\mathcal{L}[\tilde{x},\tilde{u}]\\
&\geq \int_{a}^{b}\Bigl[-\tilde{p}(t)g_{x^{\rho}}(\cdots)(x^{\rho}(t)-\tilde{x}^{\rho}(t))
-(\tilde{p}(t))^{\triangle}(x^{\rho}(t)-\tilde{x}^{\rho}(t))\\
&\qquad \qquad -\tilde{p}(t)
g_{u^{\rho}}(\cdots)(u^{\rho}(t)-\tilde{u}^{\rho}(t))
-\tilde{p}(t)g_{z}(\cdots)(x(a)-\tilde{x}(a))\\
&\qquad \qquad -\tilde{p}(t)g_{s}(\cdots)(x(b)-\tilde{x}(b))\Bigr]\nabla t\\
&\qquad + \tilde{p}(b)(x(b)-\tilde{x}(b))-\tilde{p}(a)(x(a)-\tilde{x}(a)),
\end{split}
\end{equation*}
where$(\cdots)=
(t,\tilde{x}^{\rho}(t),\tilde{u}^{\rho}(t),\tilde{x}(a),\tilde{x}(b)).$
Integrating by parts the term in $(\tilde{p})^{\triangle}$  we
get
\begin{multline*}
\mathcal{L}[x,u]-\mathcal{L}[\tilde{x},\tilde{u}]\geq
\int_{a}^{b}\tilde{p}(t)\Bigl[x^{\nabla}(t)-\tilde{x}^{\nabla}(t)
-g_{x^{\rho}}(\cdots)(x^{\rho}(t)-\tilde{x}^{\rho}(t))\\
- g_{u^{\rho}}(\cdots)(u^{\rho}(t)-\tilde{u}^{\rho}(t))
-g_{z}(\cdots)(x(a)-\tilde{x}(a))-g_{s}(\cdots)(x(b)-\tilde{x}(b))\Bigr]\nabla t.
\end{multline*}
Using \eqref{ham1} we obtain
\begin{equation*}
\begin{split}
\mathcal{L}[x,u]&-\mathcal{L}[\tilde{x},\tilde{u}]\\
&\geq
\int_{a}^{b}\tilde{p}(t)\Bigl[g(t,x^{\rho}(t),u^{\rho}(t),x(a),x(b))
-g(t,\tilde{x}^{\rho}(t),\tilde{x}^{\rho}(t),\tilde{x}(a),\tilde{x}(b))\\
&\qquad\qquad\qquad -g_{x^{\rho}}(\cdots)(x^{\rho}(t)-\tilde{x}^{\rho}(t))-
g_{u^{\rho}}(\cdots)(u^{\rho}(t)-\tilde{u}^{\rho}(t))\\
&\qquad\qquad\qquad
-g_{z}(\cdots)(x(a)-\tilde{x}(a))
-g_{s}(\cdots)(x(b)-\tilde{x}(b))\Bigr]\nabla t.
\end{split}
\end{equation*}
Note that the integrand is positive due to $\tilde{p}(t)\geq 0$ for
all $t\in[a,b]$ and joint convexity of $g$ in
$(x^{\sigma},u^{\sigma},z,s)$. We conclude that
$\mathcal{L}[x,u]\geq\mathcal{L}[\tilde{x},\tilde{u}]$ for each
admissible pair $(x,u)$.
\end{proof}

\begin{Example}
Consider the problem \eqref{ham:nc:rp} in Example~\ref{ham:nc:R}.
The integrand is independent of $(x,z,s)$ and convex in $u$. The
right-hand side of the control system is linear in $(u,z,s)$ and
independent of $x$. Hence,
\begin{equation*}
\begin{split}
\tilde{x}(t)&=t^{2}\tilde{x}(1), \quad \tilde{x}(1)\in \R,\\
\tilde{u}(t)&=0
\end{split}
\end{equation*}
gives, by Theorem~\ref{sc}, the global minimum to the problem.
\end{Example}

\begin{Example}
Consider again the problem from Example~\ref{lagrangform}. Replacing
$x^{\nabla}$ by $u^{\rho}$ we can rewrite problem \eqref{ex:1} as
\begin{equation*}
 \text{minimize} \quad   \mathcal{L}[x,u]=\int_{0}^{1}\left( (u^{\rho}(t))^{2}+\alpha x^2(0)+\beta(x(1)-1)^2\right)\nabla t
\end{equation*}
subject to
$x^{\nabla}(t)=u^{\rho}(t)$.
Function $f$ is independent of $x$ and convex in $(u,z,s)$. The
right-hand side of the control system is linear in $u$ and
independent of $(x,z,s)$. Therefore, $\tilde{x}(t)=c(\alpha, \beta) t+\tilde{x}(0,\alpha, \beta)$ is, by Theorem~\ref{sc},
a global minimizer of the problem.
\end{Example}


\section*{Acknowledgments}

The authors were partially supported by the
\emph{Center for Research and Development in Mathematics and Applications}
(CIDMA) of University of Aveiro via FCT and the EC fund FEDER/POCI 2010.
ABM was also supported by Bia{\l}ystok University of Technology
via a project of the Polish Ministry of Science and Higher Education
"Wsparcie miedzynarodowej mobilnosci naukowcow";
DFMT by the Portugal--Austin project UTAustin/MAT/0057/2008.



\end{document}